\documentclass[12pt,a4paper]{amsart}
\usepackage{amssymb}
\usepackage{fullpage}
\usepackage{url}
\usepackage[table]{xcolor}

\newcommand{\CC}{{C\nolinebreak[4]\hspace{-.05em}\raisebox{.4ex}{\tiny\bf ++}}}

\newtheorem{theorem}{Theorem}[section]
\newtheorem{corollary}[theorem]{Corollary}

\newtheorem{lemma}[theorem]{Lemma}

\theoremstyle{definition}
\newtheorem{definition}{Definition}[section]
\theoremstyle{definition}
\newtheorem{problem}[theorem]{Problem}
\theoremstyle{definition}
\newtheorem{example}[theorem]{Example}
\theoremstyle{definition}
\newtheorem{rem}{Remark}[section]

\title{Orderly generation of Butson Hadamard matrices}
\author{Pekka H.J. Lampio, Patric R.J. \"Osterg\aa rd, and Ferenc Sz\"oll\H{o}si}
\date{\today. Preprint. This research was supported in part by the Academy of Finland, Grant \#289002}
\address{P.H.J. L., P.R.J. \"O., and F. Sz.: Department of Communications and Networking, Aalto University School of Electrical Engineering, P.O. Box 15400, 00076 Aalto, Finland}
\email{pekka.lampio@aalto.fi, patric.ostergard@aalto.fi, szoferi@gmail.com}

\begin{document}
\begin{abstract}
In this paper Butson-type complex Hadamard matrices $\mathrm{BH}(n,q)$ of order $n$ and complexity $q$ are classified for small parameters by computer-aided methods. Our main results include the enumeration of $\mathrm{BH}(21,3)$, $\mathrm{BH}(16,4)$, and $\mathrm{BH}(14,6)$ matrices. There are exactly $72$, 1786763, and $167776$ such matrices, up to monomial equivalence. Additionally, we show an example of a $\mathrm{BH}(14,10)$ matrix for the first time, and show the nonexistence of $\mathrm{BH}(8,15)$, $\mathrm{BH}(11,q)$ for $q\in\{10,12,14,15\}$, and $\mathrm{BH}(13,10)$ matrices.
\end{abstract}
\maketitle

\section{Introduction}
Let $n$ and $q$ be positive integers. A Butson-type complex Hadamard matrix of order $n$ and complexity $q$ is an $n\times n$ matrix $H$ such that $HH^\ast=nI_n$, and each entry of $H$ is some complex $q$th root of unity, where $I_n$ denotes the identity matrix of order $n$, and $H^\ast$ denotes the conjugate transpose of $H$. The rows (and columns) of $H$ are therefore pairwise orthogonal in $\mathbb{C}^n$. For a fixed $n$ and $q$ we denote the set of all Butson-type complex Hadamard matrices by $\mathrm{BH}(n,q)$, and we simply refer to them as a ``Butson matrix'' for brevity \cite{cHOR}. The canonical examples are the Fourier matrices $F_n:=[\mathrm{exp}(2\pi\mathbf{i} jk/n)]_{j,k=1}^n\in\mathrm{BH}(n,n)$, frequently appearing in various branches of mathematics \cite{cKarol}.

A major unsolved problem in design theory is ``The Hadamard Conjecture'' which predicts the existence of $\mathrm{BH}(n,2)$ matrices (real Hadamard matrices) for all orders divisible by $4$. The concept of Butson matrices was introduced to shed some light onto this question from a more general perspective \cite{cBUT}. Complex Hadamard matrices play an important role in the theory of operator algebras \cite{cHaa}, \cite{cNic}, and they have also applications in harmonic analysis \cite{cKMat}. Currently there is a renewed interest in complex Hadamard matrices due to their connection to various concepts of quantum information theory, e.g., to quantum teleportation schemes and to mutually unbiased bases \cite{cBAN}, \cite{cDIT}, \cite{cKA}, \cite{cKarol}, \cite{cWer}. 

This paper is concerned with the computer-aided generation and classification of Butson matrices. Let $X$ be an $n\times n$ monomial matrix, that is $X$ has exactly one nonzero entry in each of its rows and columns which is a complex $q$th root of unity. The group $G$ of pairs of monomial matrices act on the Butson matrix $H$ by $H^{(X,Y)}\to XHY^\ast$. Two Butson matrices $H_1$ and $H_2$ are called (monomial) equivalent, if they are in the same $G$-orbit. The automorphism group of $H$, denoted by $\mathrm{Aut}(H)$ is the stabilizer subgroup of $G$ with respect to $H$. Note that if $H\in\mathrm{BH}(n,q)$ then naturally $H\in\mathrm{BH}(n,r)$ for any $r$ being a multiple of $q$. Therefore the group $\mathrm{Aut}(H)$ depends on the choice of $q$.

Earlier work predominantly considered the classification of the real case in a series of papers \cite{cKha}, \cite{cKIM}, \cite{cSPE}, see also \cite[Section~7.5]{OLDBOOK} for a historical overview. The quaternary case also received some attention in \cite{cLOS} and \cite{cS1}. Other papers in the literature dealt with settling the simpler existence problem through combinatorial constructions \cite{cBAN}, \cite{cSEB}, \cite{cS2}, \cite{cKYO} or focused on the generation of matrices with some special structure \cite{cAKI}, \cite{cCCdL}, \cite{cCHK}, \cite{cDJ}, \cite{cPAD}, \cite{cHIR}, \cite{cMW}.

The outline of this paper is as follows. In Section~\ref{sect2} we give a short overview of computer representation of Butson matrices, and recall the concept of vanishing sums of root of unity. In~Section~\ref{sect3} we briefly describe the method of orderly generation which serves as the framework used for equivalence-free exhaustive generation. In Section~\ref{sect4} we present three case studies: the classification of $\mathrm{BH}(16,4)$ matrices; the classification of $\mathrm{BH}(21,3)$ matrices; and the nonexistence of $\mathrm{BH}(n,q)$ matrices for several values $n$ and $q$. An additional notable contribution of this section is Theorem~\ref{newkron} establishing a connection between unreal $\mathrm{BH}(n,6)$ matrices and $\mathrm{BH}(2n,4)$ matrices. We conclude the paper in Section~\ref{sect99} with several open problems.

The results of this paper considerably extend the work \cite[Theorem~7.10]{cBAN}, where the (non)existence of Butson matrices was settled for $n\leq 10$ and $q\leq 14$. The reader might wish to jump ahead to Table~\ref{tableBE} to get a quick overview of the known number of $\mathrm{BH}(n,q)$ matrices for $n\leq 21$ and $q\leq 17$, including the new results established in this paper for the first time. The generated matrices are available as an electronic supplement on the web.\footnote{See \url{https://wiki.aalto.fi/display/Butson}.} The interested reader is also referred to \cite{cKarolweb} where various parametric families of complex Hadamard matrices \cite{cDIT} can be found, based on the catalog \cite{cKarol}.

\section{Computer representation of Butson Hadamard matrices}\label{sect2}
A Butson matrix $H\in\mathrm{BH}(n,q)$ is conveniently represented in logarithmic form, that is, the matrix $H=[\mathrm{exp}(2\pi\mathbf{i}\varphi_{j,k}/q)]_{j,k=1}^n$ is represented by the matrix $L(H):=[\varphi_{j,k}\ \mathrm{mod}\ q]_{j,k=1}^n$ with the convention that $L_{j,k}\in\mathbb{Z}_q$ for all $j,k\in\{1,\dots,n\}$. Throughout this paper we denote by $\mathbb{Z}_q$ the additive group of integers modulo $q$, where the underlying set is $\{0,\dots,q-1\}$. With this convention $(\mathbb{Z}_q^n,\prec)$ is a linearly ordered set, where for $a,b\in\mathbb{Z}_q^n$ we write $a\prec b$ if and only if $a=b$ or $a$ lexicographically precedes $b$.

\begin{example}\label{ex1}
The following is a $\mathrm{BH}(14,10)$ matrix $H$, displayed in logarithmic form.
\tiny
\[\text{\normalsize$L(H)=$}\left[
\begin{array}{cccccccccccccc}
 0 & 0 & 0 & 0 & 0 & 0 & 0 & 0 & 0 & 0 & 0 & 0 & 0 & 0 \\
 0 & 0 & 0 & 2 & 2 & 2 & 4 & 4 & 5 & 6 & 6 & 7 & 8 & 8 \\
 0 & 0 & 2 & 6 & 6 & 8 & 2 & 8 & 5 & 1 & 4 & 6 & 0 & 4 \\
 0 & 0 & 4 & 8 & 8 & 6 & 2 & 5 & 3 & 6 & 8 & 2 & 4 & 0 \\
 0 & 0 & 8 & 4 & 4 & 8 & 6 & 0 & 9 & 4 & 6 & 2 & 5 & 2 \\
 0 & 1 & 6 & 5 & 1 & 4 & 8 & 9 & 2 & 7 & 2 & 7 & 3 & 6 \\
 0 & 1 & 6 & 0 & 6 & 4 & 8 & 4 & 7 & 2 & 2 & 2 & 8 & 6 \\
 0 & 5 & 2 & 8 & 3 & 5 & 4 & 9 & 0 & 3 & 0 & 5 & 7 & 8 \\
 0 & 5 & 2 & 2 & 7 & 0 & 6 & 7 & 2 & 9 & 4 & 1 & 5 & 7 \\
 0 & 5 & 4 & 6 & 1 & 0 & 0 & 5 & 8 & 5 & 3 & 9 & 7 & 2 \\
 0 & 5 & 7 & 4 & 9 & 6 & 4 & 3 & 2 & 1 & 8 & 9 & 9 & 4 \\
 0 & 5 & 8 & 0 & 5 & 2 & 9 & 7 & 4 & 3 & 8 & 7 & 3 & 2 \\
 0 & 6 & 0 & 8 & 9 & 4 & 6 & 2 & 7 & 8 & 4 & 4 & 2 & 2 \\
 0 & 6 & 5 & 3 & 4 & 9 & 1 & 2 & 7 & 8 & 9 & 4 & 2 & 7 \\
\end{array}
\right]\text{\normalsize, $|\mathrm{Aut}(H)|=20$.}\]
\normalsize
\end{example}

Observe that the matrix shown in Example~\ref{ex1} is in dephased form \cite{cKarol}, that is, its first row and column are all $0$ (representing the logarithmic form of $1$). Every matrix can be dephased by using equivalence-preserving operations. Throughout this paper all matrices are assumed to be dephased.

Let $H\in\mathrm{BH}(n,q)$, and let $r_1, r_2\in\mathbb{Z}_q^n$ be row vectors of $L(H)$. Then, by complex orthogonality, the difference row $d:=r_1-r_2\in\mathbb{Z}_q^n$ satisfies $\mathcal{E}_{n,q}(d)=0$, where
\begin{equation*}
\mathcal{E}_{n,q} \colon \mathbb{Z}_q^n \to \mathbb{C},\qquad \mathcal{E}_{n,q}(x):=\sum_{i=1}^{n}\mathrm{exp}(2\pi\mathbf{i}x_i/q)
\end{equation*}
 is the evaluation function. In other words, $d$ represents an $n$-term vanishing sum of $q$th roots of unity \cite{cLL}. We note that the number $\mathcal{E}_{n,q}(x)$ is algebraic, and its value is invariant up to permutation of the coordinates of $x\in\mathbb{Z}_q^n$. In particular, $\mathcal{E}_{n,q}(x)=\mathcal{E}_{n,q}(\mathrm{Sort}(x))$, where $\mathrm{Sort}(x)=\min\{\sigma(x)\colon \text{$\sigma$ is a permutation on $n$ elements}\}$ (with respect to the ordering $\prec$). We introduce the orthogonality set which contains the representations of the normalized, sorted, $n$-term vanishing sums of $q$th roots of unity:
\begin{equation*}
\mathcal{O}(n,q):=\{x\in\mathbb{Z}_q^n \colon x_1=0;\ x=\mathrm{Sort}(x);\ \mathcal{E}_{n,q}(x)=0\}.
\end{equation*}
Once precomputed, the set $\mathcal{O}(n,q)$ allows us to determine if two rows of length $n$ of a dephased matrix with elements in $\mathbb{Z}_q$ are complex orthogonal in a combinatorial way, i.e., without relying on the analytic function $\mathcal{E}_{n,q}$. Indeed, for any vector $x\in \mathbb{Z}_q^n$ having at least one $0$ coordinate, $\mathcal{E}_{n,q}(x)=0$ if and only if $\mathrm{Sort}(x)\in\mathcal{O}(n,q)$.

One can observe that for certain values of $n$ and $q$ the set $\mathcal{O}(n,q)$ is empty, that is, it is impossible to find a pair of orthogonal rows in $\mathbb{Z}_q^n$ and consequently $\mathrm{BH}(n,q)$ matrices do not exist. For example, it is easy to see that $|\mathcal{O}(n,2)|=0$ for odd $n>1$. The following recent result characterizes the case when the set $\mathcal{O}(n,q)$ is nonempty, and should be viewed as one of the fundamental necessary conditions on the existence of Butson matrices. 
\begin{theorem}[\mbox{\cite[Theorem~5.2]{cLL}}]\label{LLMAIN}
Let $n$, $r$, and $a_i$, $i\in\{1,\dots, r\}$ be positive integers, and let $q=\prod_{i=1}^rp_i^{a_i}$ with distinct primes $p_i$, $i\in\{1,\dots, r\}$. Then, we have $|\mathcal{O}(n,q)|\geq 1$ if and only if there exist nonnegative integers $w_i$, $i\in\{1,\dots,r\}$ such that $n=\sum_{i=1}^r w_i p_i$.
\end{theorem}

In order to classify all $\mathrm{BH}(n,q)$ matrices for a given parameters, three tasks have to be completed: the set $\mathcal{O}(n,q)$ has to be determined; vectors $x\in \mathbb{Z}_q^n$ orthogonal to a prescribed set of vectors should be generated; and equivalent matrices should be rejected. In the next section we discuss these three tasks in detail.

\section{Generating Butson Hadamard matrices}\label{sect3}
\subsection{Generating the vanishing sums of roots of unity}\label{sssectonq}
For a given $n$ and $q$, our first task is to determine the set $\mathcal{O}(n,q)$ which in essence encodes complex orthogonality of a pair of rows. It turns out that when $q$ is a product of at most two prime powers, then a compact description of the elements of $\mathcal{O}(n,q)$ is possible. The following two results are immediate consequences of \cite[Corollary~3.4]{cLL}.

\begin{lemma}\label{l1}
Let $a$, $n$ be positive integers, and let $q=p^a$ be a prime power. Let $u=[0,q/p,2q/p,\dots,(p-1)q/p]\in\mathbb{Z}_q^p$, and let $x\in\mathbb{Z}_q^n$. Then $x\in\mathcal{O}(n,q)$ if and only if there exist a positive integer $s$ such that $ps=n$, and $r_i\in\{0,\dots,q/p-1\}$, $i\in\{1,\dots, s-1\}$, such that $x=\mathrm{Sort}([u,r_1+u,\dots,r_{s-1}+u])$.
\end{lemma}

\begin{lemma}\label{l2}
Let $a$, $b$ and $n$ be positive integers, and let $q=p_1^ap_2^b$ be the product of two distinct prime powers. Let $u=[0,q/p_1,2q/p_1,\dots,(p_1-1)q/p_1]\in\mathbb{Z}_q^{p_1}$, $v=[0,q/p_2,2q/p_2,\dots,(p_2-1)q/p_2]\in\mathbb{Z}_q^{p_2}$, and let $x\in\mathbb{Z}_q^n$. Then $x\in\mathcal{O}(n,q)$ if and only if there exist nonnegative integers $s$, $t$ such that $p_1s+p_2t=n$, and $r_i\in\{0,\dots,q/p_1-1\}$, $i\in\{1,\dots, s\}$, $R_j\in\{0,\dots,q/p_2-1\}$, $j\in\{1,\dots, t\}$ such that $x=\mathrm{Sort}([r_1+u,r_2+u,\dots,r_s+u,R_1+v,R_2+v,\dots,R_t+v])$, and $0\in\{r_1,R_1\}$.
\end{lemma}
The main point of the rather technical Lemma~\ref{l1} and Lemma~\ref{l2} is the following: as long as $q$ is the product of at most two prime powers, the constituents of any $n$-term vanishing sum of $q$th roots of unity are precisely $p$-term vanishing sums, where $p$ is some prime divisor of $q$. These $p$-term vanishing sums are in turn the (scalar multiplied, or, ``rotated'') sums of every $p$th root of unity.

The significance of these structural results is that based on them one can design an efficient algorithm to generate the set $\mathcal{O}(n,q)$ as long as $q<30=2\cdot3\cdot5$ in a combinatorial way (i.e., without the need of the analytic function $\mathcal{E}_{n,q}$). In particular, this task can be done relying on exact integer arithmetic. We spare the reader the details.

In certain simple cases it is possible to enumerate (as well as to generate) the set $\mathcal{O}(n,q)$ by hand. We offer the following counting formulae for means of checking consistency.
\begin{lemma}\label{minor1}
Let $a$ and $n$ be positive integers, and let $q=p^a$ be a prime power. Assume that $p$ divides $n$. Then $|\mathcal{O}(n,q)|=\binom{(n+q)/p-2}{n/p-1}$.
\end{lemma}
\begin{proof}
By Lemma~\ref{l1} members of the set $\mathcal{O}(n,q)$ can be partitioned into $n/p$ parts of the form $r_i+[0,q/p,2q/p,\dots,(p-1)q/p]$, each part being identified by the rotation $r_i\in\{0,\dots,q/p-1\}$, $i\in\{0,\dots, n/p-1\}$ with $r_0=0$. The number of ways to assign $q/p$ values to a set of $n/p-1$ variables (up to relabelling) is exactly $\binom{(n+q)/p-2}{n/p-1}$; each of these choices lead to different members of $\mathcal{O}(n,q)$.
\end{proof}
A slightly more complicated variant is the following result.
\begin{lemma}\label{minor2}
Let $n\geq 2$ be an integer, let $p$ be an odd prime, and let $q=2p$. Then
\[|\mathcal{O}(n,q)|=\frac{1+(-1)^n}{2}\binom{p+\left\lfloor n/2\right\rfloor-2}{\left\lfloor n/2\right\rfloor-1}+\sum_{\substack{2s+pt=n\\ s\geq1,t\geq1}}\binom{p+s-1}{s}+\sum_{\substack{2s+pt=n\\ s\geq1, t\geq1}}\binom{p+s-2}{s-1}+\delta,\]
where $\delta=1$ if $p$ divides $n$, and $\delta=0$ otherwise.
\end{lemma}
\begin{proof}
This can be inferred by using Lemma~\ref{l2}. We count the elements $x\in\mathcal{O}(n,q)$ based on how many pairs of coordinates $[x_i,x_i+p]\in\mathbb{Z}_q^2$ they have. Let us call this number $s$.

If $s=0$, then clearly $p$ divides $n$ and $x$ can be partitioned into $t=n/p$ parts, each being either of the form $[0,2,4,\dots,2p-2]$ or $[1,3,5,\dots,2p-1]$. However, since $s=0$, only one of these two forms could appear, and since $x$ must have a coordinate $0$, this left us with only $\delta=1$ case.

If $s=n/2\geq 1$ then $n$ is necessarily even, and $x$ can be partitioned into $n/2$ parts, each being of the form $[x_i,x_i+p]$ for some $x_i\in\{0,\dots,p-1\}$, $i\in\{1,\dots,n/2\}$. Since $x$ must contain $0$, one of these parts must be $[0,p]$, while the other $n/2-1$ parts can take $p$ different forms. There are a total of $\binom{p+n/2-2}{n/2-1}$ cases.

Finally, if $0<s<n/2$, then there are either $t=(n-2s)/p\geq 1$ parts of the form $[0,2,4,\dots,2p-2]$, or $t$ parts of the form $[1,3,5,\dots,2p-1]$. In the first case there are $\binom{p+s-1}{s}$ ways to assign values to the remaining $s$ parts; in the second case, since $x$ must have a $0$ coordinate, there are $\binom{p+s-2}{s-1}$ ways to assign values to the remaining $s$ parts.
\end{proof}
The statements of Lemma~\ref{minor1} and Lemma~\ref{minor2} are strong enough to cover all cases $q\leq 17$ except for $q\in\{12,15\}$. We have applied these results to verify that the computer-generated sets $\mathcal{O}(n,q)$ are of the correct cardinality. In the next section we will see a further application of the set $\mathcal{O}(n,q)$.
\begin{rem}\label{magicsumq}
There is no analogous result to Lemma~\ref{l1} and Lemma~\ref{l2} when $q$ has more than two prime factors. Indeed, the reader might amuse themselves by verifying that while $[0,1,7,13,19,20]\in\mathcal{O}(6,30)$, it does not have any $m$-term vanishing subsums with $m\in\{2,3,5\}$. See \cite[Example~6.7]{cLL} for examples of similar flavor.

An alternative, algebraic way to generate the set $\mathcal{O}(n,q)$ is to compute for all $x\in\mathbb{Z}_q^n$ with $x_1=0$ and $\mathrm{Sort}(x)=x$ the minimal polynomial $p(t)$ of the algebraic number $\mathcal{E}_{n,q}(x)$. With this terminology, $x\in\mathcal{O}(n,q)$ if and only if $p(t)=t$. The efficiency of this approach can be greatly improved by testing first by fast numerical means whether the Euclidean norm of $\mathcal{E}_{n,q}(x)$ is small, say if $\left\|\mathcal{E}_{n,q}(x)\right\|^2=\mathcal{E}_{n,q}(x)\mathcal{E}_{n,q}(-x)<0.01$ holds.
\end{rem}
\subsection{Orderly generation of rectangular matrices}\label{subsorder}
In this section we briefly recall the method of orderly generation, which is a technique for generating matrices exhaustively in a way that no equivalence tests between different matrices are required \cite[Section~4.2.2]{cKAS}, \cite{cREA}. Such a search can be efficiently executed in parallel. The main idea is to select from each equivalence class of Butson matrices a canonical representative, and organize the search in a way to directly aim for this particular matrix. Variations of this basic approach were employed for the classification of $\mathrm{BH}(n,2)$ matrices for $n\leq 32$, see \cite{cKha}, \cite{cSPE}.

Let $n,r\geq1$. We associate to each $r\times n$ matrix $R$ whose elements are complex $q$th roots of unity its vectorization $v(R):=[L(R)_{1,1}, \dots, L(R)_{1,n}, L(R)_{2,1},\dots,L(R)_{r,n}]\in\mathbb{Z}_q^{rn}$ formed by concatenating the rows of its logarithmic form $L(R)$. We say that $R$ is in canonical form, if $v(R)=\min\{v(XRY^\ast)\colon \text{$X$ and $Y$ are $q$th root monomial matrices}\}$, where comparison is done with respect to the ordering $\prec$. Canonical matrices defined in this way have a number of remarkable properties. For example, if $R$ is canonical, and $r_1$ and $r_2$ are consecutive rows of $L(R)$, then $r_1\prec r_2$, and analogously for the columns. Moreover, canonical matrices are necessarily dephased. Let $\sigma$ be a permutation on $r$ elements, and let $i\in\{1,\dots,n\}$. Let us denote by $R^{(\sigma,i)}$ the matrix which can be obtained from $R$ by permuting its rows according to $\sigma$, then swapping its first and $i$th columns, then dephasing it, and finally arranging its columns according to $\prec$. 
\begin{lemma}\label{l35}
Let $n,r\geq 1$, and let $R$ be an $r\times n$ matrix. The matrix $R$ is canonical, if and only if $v(R)=\min\{v(R^{(\sigma,i)})\colon\text{$\sigma$ is a permutation on $r$ elements, $i\in\{1,\dots,n\}$}\}$.
\end{lemma}
\begin{proof}
This is an immediate consequence of the fact that canonical matrices are dephased and their columns are sorted with respect to $\prec$.
\end{proof}
It is possible to further improve the test described in Lemma~\ref{l35} by the following considerations. Let $k\in\{1,\dots,r\}$ and let $R_k$ denote the leading $k\times n$ submatrix of $R$. If there exists a pair $(\sigma,i)$ such that $v(R_k)\neq v(R^{(\sigma,i)}_k)$ and $v(R_k)\prec v(R^{(\sigma,i)}_k)$ then the same holds for all other permutations whose first $k$ coordinates agree with that of $\sigma$. In particular, all those permutations can be skipped. An efficient algorithm for permutation generation with restricted prefixes is discussed in \cite[Algorithm~X]{cKNU}.
 
The computational complexity of this method is exponential in the number of rows $r$, polynomial in the number of columns $n$, and independent of the complexity $q$. Testing whether a matrix is in canonical form is the most time-consuming part of the generation.

Finally, we note one more property of canonical matrices.
\begin{lemma}\label{lsecpr}
Let $H\in\mathrm{BH}(n,q)$ in canonical form. Let us denote by $r_2$ the second row of $L(H)$, and by $c_2$ the second column of $L(H)$. Then $r_2\in\mathcal{O}(n,q)$ and $c_2^T\in\mathcal{O}(n,q)$.
\end{lemma}
\begin{proof}
This follows from the fact that $H$ is necessarily dephased, and its rows and columns are ordered with respect to the ordering $\prec$.
\end{proof}
The significance of Lemma~\ref{lsecpr} is that if the (transpose of the) logarithmic form of the second column of a rectangular orthogonal matrix is not a prefix of any of the elements of the set $\mathcal{O}(n,q)$, then that matrix can be discarded during the search. We refer to this look-ahead strategy as ``pruning the search tree by the second column condition''.

The matrices $H\in\mathrm{BH}(n,q)$ (more precisely, their logarithmic form) are generated in a row-by-row fashion. Every time a new row is appended we first test whether it is orthogonal to all previous rows by checking if the difference vectors belong to the set $\mathcal{O}(n,q)$ as described in Section~\ref{sssectonq}. If the rows of the matrix are pairwise orthogonal, then we further check whether (the transpose of) its second column is a prefix of an element of the set $\mathcal{O}(n,q)$. Finally, we test whether it is in canonical form. Only canonical matrices will be processed further, the others will be discarded and backtracking takes place.

\begin{rem}
In a prequel to this work \cite{cLOS} we employed the method of canonical augmentation \cite[Section~4.2.3]{cPET} to solve the more general problem of classification of all rectangular orthogonal matrices. Here we solve the relaxed problem of classification of those matrices which can be a constituent of an orderly-generated Butson matrix. The reader might wish to look at the impact of the second column pruning strategy on the number of $r\times 14$ submatrices in Table~\ref{tabletreecomp}, where we compare the size of the search trees encountered with these two methods during the classification of $\mathrm{BH}(14,4)$ matrices.
\end{rem}

\tiny
\begin{table}[htbp]
\[\begin{array}{r|rrrr}
\hline
r && \text{Total}&& \mathrm{BH}(14,4)\\
\hline
   1 &&          1 &&         1\\
	 2 &&          4 &&         4\\
	 3 &&         42 &&        42\\
   4 &&     10141 &&     9142\\
   5 &&  1601560 &&   637669\\
   6 && 21311746 && 2118948\\
   7 && 17175324 &&   189721\\
   8 &&  4234669 &&   155777\\
   9 &&  1675882 &&   108598\\
  10 &&    716604 &&    56103\\
  11 &&    249716 &&    17992\\
  12 &&     62739 &&     5558\\
  13 &&      9776 &&     3039\\
  14 &&        752 &&       752\\
	\hline
\end{array}\]
\caption{Comparison of the size of the search trees.}\label{tabletreecomp}
\end{table}
\normalsize

\begin{rem}
We have observed earlier that the computational cost of equivalence testing is independent of the complexity $q$ when orderly generation is used. This is in contrast with the method of canonical augmentation employed earlier in \cite{cLOS} which relies on graph representation of the $r\times n$ rectangular orthogonal matrices with $q$th root entries on $3q(r+n)+r$ vertices. See \cite{cLAM}, \cite{pekkadiffm} for more on graph representation of Butson matrices.
\end{rem}

\subsection{Augmenting rectangular orthogonal matrices}
Let $n,r\geq 1$, and let $R$ be an $r\times n$ canonical matrix with pairwise orthogonal rows. Let $r_i$, $i\in\{1,\dots, r\}$ denote the rows of $L(R)$. The goal of this section is to describe methods for generating the vectors $x\in\mathbb{Z}_q^n$ such that $\mathcal{E}_{n,q}(r_i-x)=0$ hold simultaneously for every $i\in\{1,\dots, r\}$. Note that since we are only interested in canonical Butson matrices, we assume that $x_1=0$.

The most straightforward way of generating the vectors $x$ is to consider the permutations of the elements of the set $\mathcal{O}(n,q)$. Indeed, the following two conditions (i) $x$ has a coordinate $0$; and (ii) $\mathcal{E}_{n,q}(r_1-x)=0$ are together equivalent to $\mathrm{Sort}(x)\in\mathcal{O}(n,q)$. For all such vectors $x$ the remaining conditions $\mathcal{E}_{n,q}(r_i-x)=0$, $i\in\{2,\dots,r\}$ should be verified. This strategy of generating the rows works very well for small matrices, say, up to $n\leq 11$. One advantage of this na\"ive method is that permutations can be generated one after another, without the need of excessive amount of memory \cite{cKNU}.

Next we describe a more efficient divide-and-conquer strategy \cite[p.~157]{cKAS} for generating the vectors $x$. Let $m\in\{1,\dots,n-1\}$ be a parameter, and for every $i\in\{1,\dots, r\}$ write $r_i=[a_i,b_i]$, where $a_i\in\mathbb{Z}_q^{n-m}$, $b_i\in\mathbb{Z}_q^{m}$, and write $x=[c,d]$, where $c\in\mathbb{Z}_q^{n-m}$, $d\in\mathbb{Z}_q^{m}$.

As a first step, we create a lookup table $\mathcal{T}$ indexed by $\iota\in\mathbb{C}^r$, where the value at $\mathcal{T}(\iota)$ is a certain subset of $\mathbb{Z}_q^m$. Formally, consider $\mathcal{T}\colon\mathbb{C}^r\to\mathcal{P}(\mathbb{Z}_q^m)$, where for every $d\in\mathbb{Z}_q^m$ it holds that $d\in\mathcal{T}([\mathcal{E}_{n,q}(b_1-d),\dots,\mathcal{E}_{n,q}(b_r-d)])$. Naturally, we assume that the values form a partition of $\mathcal{P}(\mathbb{Z}_q^m)$. As a second step, for every $c\in\mathbb{Z}_q^{n-m}$ we look the vectors $d\in\mathbb{Z}_q^m$ up (if any) contained in the set $\mathcal{T}([-\mathcal{E}_{n,q}(a_1-c),\dots,-\mathcal{E}_{n,q}(a_r-c)])$. By construction, the vectors $x=[c,d]$ fulfill the desired conditions; if no such $d$ were found, then $c$ cannot be a prefix of $x$.

In practice, however, it is inconvenient to work with complex-valued indices, and therefore one needs to use a hash function $\mathcal{H}\colon \mathbb{C}^r\to \mathbb{Z}^+_0$ to map them to nonnegative integers. This leads to a convenient implementation at the expense of allowing hash collisions to occur. Since it is not at all clear how to come up with a nontrivial hash function (apart from $\mathcal{H}\equiv 0$) we describe here an elegant choice exploiting the number theoretic properties of the Gaussian- and the Eisenstein integers. We assume for the following argument that $q\in\{2,3,4,6\}$. Recall that $\mathcal{T}$ was indexed by complex $r$-tuples of the form $[\mathcal{E}_{n,q}(b_1-d),\dots,\mathcal{E}_{n,q}(b_r-d)]$. Let $p_{\mathrm{big}}$ be a (large) prime, and let $p_i\ll p_{\mathrm{big}}$, $i\in\{1,\dots, r\}$ be $r$ other distinct primes. We define $\mathcal{H}$ through the Euclidean norm of the partial inner products as follows: $\mathcal{H}([\mathcal{E}_{n,q}(b_1-d),\dots,\mathcal{E}_{n,q}(b_r-d)]):=\sum_{i=1}^r\left\|\mathcal{E}_{n,q}(b_i-d)\right\|^2p_i\ (\mathrm{mod}\ p_{\mathrm{big}})$. This gives rise to a table $\mathcal{S}\colon\mathbb{Z}_0^+\to\mathcal{P}(\mathbb{Z}_q^m)$ which is defined through $\mathcal{T}$ and $\mathcal{H}$ as follows: for every $\iota\in\mathbb{C}^r$, let $\mathcal{S}(\mathcal{H}(\iota)):=\mathcal{T}(\iota)$. As for the second step, for every $c\in\mathbb{Z}_q^{n-m}$ we look the vectors $d\in\mathbb{Z}_q^m$ up (if any) contained in the set $\mathcal{S}(k)$, $k\in\{0,\dots,p_{\mathrm{big}}-1\}$, for which the modular equation $k\equiv \sum_{i=1}^r\left\|\mathcal{E}_{n,q}(a_i-c)\right\|^{2}p_i\ (\mathrm{mod}\ p_{\mathrm{big}})$ holds. Finally, for all (if any) vectors $x=[c,d]$ one should test whether they are orthogonal to the rows of $R$.

The table $\mathcal{T}$ is generated once for every matrix $R$, and it is reused again during a depth-first-search. The advantage of this technique is that as long as $n\leq 21$ and $m\approx n/2$ the $q$-ary $m$-tuples can be generated efficiently. For higher sizes, however, precomputing and storing such a table becomes quickly infeasible due to memory constraints, and therefore one needs to carefully choose the value of $m$ in terms of $n$, $q$, and the number of processors accessing the shared memory.
\begin{rem}
Let $x\in\mathbb{Z}_q^n$, and for every $i\in\mathbb{Z}_q$ let us denote by $f_i$ the frequency distribution of the number $i$ occurring as a coordinate of $x$. We have $\|\mathcal{E}_{n,2}(x)\|^2=(f_0-f_1)^2$; $\|\mathcal{E}_{n,3}(x)\|^2=f_0^2+f_1^2+f_2^2-f_0f_1-f_0f_2-f_1f_2$; $\|\mathcal{E}_{n,4}(x)\|^2=(f_0-f_2)^2+(f_1-f_3)^2$; and finally, $\|\mathcal{E}_{n,6}(x)\|^2=(f_0-f_3)^2+(f_4-f_1)^2+(f_5-f_2)^2-(f_0-f_3)(f_4-f_1)-(f_0-f_3)(f_5-f_2)-(f_4-f_1)(f_5-f_2)$. In particular, these numbers are nonnegative integers.
\end{rem}
\begin{rem}
For $q\not\in\{2,3,4,6\}$ the hash function $\mathcal{H}$ should be replaced by a suitable alternative, as the quantity $\|\mathcal{E}_{n,q}(x)\|^2$ is no longer guaranteed to be an integer. For example, when $q=10$, one may verify that for every $x\in\mathbb{Z}_{10}^n$ we have $2\|\mathcal{E}_{n,10}(x)\|^2=A+\sqrt{5}B$, where $A$ and $B$ are integers. Therefore one can map $\|\mathcal{E}_{n,10}(x)\|^2$ to $A^2+pB^2$ (where $p$ is some large prime). Similar techniques work for certain other values of $q$.
\end{rem}
\section{Results and case studies}\label{sect4}
\subsection{Main results and discussion}
Based on the framework developed in Sections~\ref{sect2}--\ref{sect3} we were able to enumerate the set $\mathrm{BH}(n,q)$ for $n\leq 11$ and $q\leq 17$ up to monomial equivalence (cf.\ \cite[Theorem~7.10]{cBAN}). Several additional cases were also settled.
\begin{theorem}
The known values of the exact number of $\mathrm{BH}(n,q)$ matrices, up to monomial equivalence, is displayed in Table~\ref{tableBE}.
\end{theorem}

\tiny
\begin{table}[htbp]
\begin{tabular}{r|rrrrrrrrrrrrrrrr}
\hline
${}_n\mkern-6mu\setminus\mkern-6mu{}^q$%
   &         2 &        3   &         4    &         5 &     6       &         7 &         8   &          9 &          10 &        11 &          12 &        13 &           14 &         15 &           16 & 17\\
\hline
 2 & 1 &  & 1 &  & 1 &  & 1 &  & 1 &  & 1 &  & 1 &  & 1 &  \\
 3 &  & 1 &  &  & 1 &  &  & 1 &  &  & 1 &  &  & 1 &  &  \\
 4 & 1 &  & 2 &  & 2 &  & 3 &  & 3 &  & 4 &  & 4 &  & 5 &  \\
 5 & &  &  & 1 & 0 &  &  &  & 1 &  & 0 &  &  & 1 &  &  \\
 6 & 0 & 1 & 1 &  & 4 &  & 3 & 1 & 0 &  & 11 &  & 0 & 1 & 5 &  \\
 7 & &  &  &  & 2 & 1 &  &  & 0 &  & 4 &  & 1 &  &  &  \\
 8 & 1 &  & 15 &  & 36 &  & 143 &  & 299 &  & 756 &  & 1412 & 0 & 2807 &  \\
 9 &   & 3 &    &   & 17 &  &    & 23 &   1 &  &  65 &  & 0 & 93 &  &  \\
10 & 0 &   & 10 & 1 & 34 &  & 60 &    & 51  &  & 577 &  & 0 & 1 & 310 &  \\
11 & &  &  &  & 0 &  &  &  & 0 & 1 & 0 &  & 0 & 0 &  &  \\
12 & 1 & 2 & 319 &  & 8703 &  & 53024 & 8 & 293123 &  & \text{E} &  & \text{E} & \text{E} & \text{E} &  \\
13 & &  &  &  & 436 &  &  &  & 0 &  & \text{E} & 1 & \text{U} & \text{U} &  &  \\
14 & 0 &  & 752 &  & 167776 & 3 & \text{E} &  & \text{E} &  & \text{E} &  & \text{E} & \text{U} & \text{E} &  \\
15 & & 0 &  & 0 & 0 &  &  & 0 & \text{U} &  & \text{U} &  & \text{U} & \text{E} &  &  \\
16 & 5 &  & 1786763 &  & \text{E} &  & \text{E} &  & \text{E} &  & \text{E} &  & \text{E} & \text{U} & \text{E} &  \\
17 &  &  &  &  & 0 &  &  &  & \text{U} &  & \text{U} &  & \text{U} & \text{U} &  & 1 \\
18 & 0 & 85 & \text{E} &  & \text{E} &  & \text{E} & \text{E} & \text{U} &  & \text{E} &  & \text{U} & \text{U} & \text{E} &  \\
19 &  &  &  &  & \text{E} &  &  &  & \text{U} &  & \text{E} &  & \text{U} & \text{U} &  &  \\
20 & \text{3} &  & \text{E} & \text{E} & \text{E} &  & \text{E} &  & \text{E} &  & \text{E} &  & \text{E} & \text{E} & \text{E} &  \\
21 &  & 72 &  &  & \text{E} & 0 &  & \text{E} & \text{U} &  & \text{E} &  & 0 & \text{E} &  &  \\
\hline
\end{tabular}
\caption{The number of $\mathrm{BH}(n,q)$ matrices up to monomial equivalence.}\label{tableBE}
\end{table}
\normalsize

The legend for Table~\ref{tableBE} is as follows. An entry in the table at position $(n,q)$ indicates the known status of the existence of $\mathrm{BH}(n,q)$ matrices. Empty cells indicate cases where $\mathrm{BH}(n,q)$ matrices do not exist by Theorem~\ref{LLMAIN}; cells marked by an ``E'' indicate cases where $\mathrm{BH}(n,q)$ matrices are known to exist, but no full classification is available; cells marked by an ``U'' indicate that existence is unknown; finally cells displaying a number indicate the exact number of $\mathrm{BH}(n,q)$ matrices up to monomial equivalence.

Next we briefly review the contents of Table~\ref{tableBE}, and comment on the cases based on their complexity $q\in\{2,3,\dots, 17\}$. We note that most of the numbers shown are new.

\underline{$q=2$}: This is the real Hadamard case. Complete classification is available up to $n\leq 32$, see \cite[Section~7.5]{OLDBOOK}, \cite{cKha}. The number of $\mathrm{BH}(36,2)$ matrices is at least $1.8\times 10^7$ \cite{cORR}, while according to \cite{cLLT} the number of $\mathrm{BH}(40,2)$ matrices is at least $3.66\times 10^{11}$.

\underline{$q=3$}: Complete classification is available up to $n\leq 21$, see Section~\ref{bh213x}. The case $\mathrm{BH}(18,3)$ was reported in \cite{cHAR2} and independently in \cite{cLAM}. Several cases of $\mathrm{BH}(21,3)$ were found by Brock and Murray as reported in \cite{cAKI} along with additional examples. There are no $\mathrm{BH}(15,3)$ matrices \cite{cHAR2}, \cite[Theorem~6.65]{OLDBOOK}, \cite[Theorem~3.2.2]{cLAM}.

\underline{$q=4$}: Classification is known up to $n\leq 16$, see \cite{cLOS}, \cite{cS1} and Section~\ref{seccase1}. The difference matrices over $\mathbb{Z}_4$ with $\lambda=4$ (essentially: the $\mathrm{BH}(16,4)$ matrices of type-$4$) were reported independently in \cite{cGM}, \cite{cHLT}, \cite{pekkadiffm}. A $\mathrm{BH}(18,4)$ can be constructed from a symmetric conference matrix \cite[Theorem~3]{cSEB}, \cite{cTUR}.

\underline{$q=5$}: An explicit example of $\mathrm{BH}(20,5)$ can be found in \cite{cSEB2}, while a $\mathrm{BH}(15,5)$ does not exist \cite[Theorem~6.65]{OLDBOOK}, \cite[Theorem~3.2.2]{cLAM}.

\underline{$q=6$}: Examples of $\mathrm{BH}(7,6)$ matrices were presented in \cite{cBRO} and independently but slightly later in \cite{cPET}. A $\mathrm{BH}(10,6)$ was reported in \cite[p.~105]{cAGA}. Several unreal $\mathrm{BH}(13,6)$ were reported in \cite{cCCdL}; additional examples were reported by Nicoar\u{a} et al.\ on the web site \cite{cKarolweb}. A $\mathrm{BH}(19,6)$ was found in \cite{cS2}, based on the approach of \cite{cPET}. A necessary condition on the existence of a $\mathrm{BH}(n,6)$ matrix comes from the determinant equation $|\mathrm{det}(H)|^2=n^n$, where the left hand side is the norm of an Eisenstein integer and therefore is of the form $A^2-AB+B^2$ for some integers $A$ and $B$ \cite{cBRO}, \cite{cWIN}. Consequently $\mathrm{BH}(n,6)$ matrices for $n\in\{5,11,15,17\}$ do not exist.

\underline{$q=7$}: The $\mathrm{BH}(14,7)$ matrices come from a doubling construction \cite{cBUT}, \cite{cKYO} while $\mathrm{BH}(21,7)$ matrices do not exist by \cite[Theorem~5]{cWIN}.

\underline{$q=8$}: Here $n=1$, or $n\geq2$ is necessarily even by Theorem~\ref{LLMAIN}. Existence follows from the existence of $\mathrm{BH}(n,4)$ matrices. A particular example of $\mathrm{BH}(6,8)$ matrix played an important role in disproving the ``Spectral Set Conjecture'' in $\mathbb{R}^3$, see \cite{cKMat}. This is one notable example of contemporary applications of complex Hadamard matrices.

\underline{$q=9$}: A $\mathrm{BH}(15,9)$ does not exist by \cite[Theorem~5]{cWIN}. 

\underline{$q=10$}: Nonexistence of $\mathrm{BH}(n,10)$ for $n\in\{6,7\}$ was proved in \cite{cBAN}. The discovery of a $\mathrm{BH}(9,10)$ matrix by Beauchamp and Nicoar\u{a} (found also independently in \cite{cKA}) was rather unexpected \cite{cKarolweb}. There are no $\mathrm{BH}(11,10)$ or $\mathrm{BH}(13,10)$ matrices (see Theorem~\ref{nonex11} and ~\ref{nonex13}). To the best of our knowledge $\mathrm{BH}(14,10)$ matrices were not known prior to this work, and Example~\ref{ex1} shows a new discovery.

\underline{$q=11$}: The Fourier matrix $F_{11}$ is unique \cite{cHIR}.

\underline{$q=12$}: A $\mathrm{BH}(5,12)$ does not exist since all $5\times 5$ complex Hadamard were shown to be equivalent to $F_5$ in \cite{cHaa}. A $\mathrm{BH}(11,12)$ does not exist by Theorem~\ref{nonex11}.

\underline{$q=13$}: The Fourier matrix $F_{13}$ is unique \cite{cHIR}.

\underline{$q=14$}: Several nonexistence results are known. The matrices $\mathrm{BH}(n,14)$ for $n\in\{6,9,10\}$ were shown to be nonexistent in \cite{cBAN}. The matrices $\mathrm{BH}(11,14)$ do not exist by Theorem~\ref{nonex11}. Finally, there are no $\mathrm{BH}(21,14)$ matrices by \cite[Theorem~5]{cWIN}.

\underline{$q=15$}: There are no $\mathrm{BH}(n,15)$ matrices for $n\in\{8,11\}$, see Theorem~\ref{nonex8} and Theorem~\ref{nonex11} respectively.

\underline{$q=16$}: Here $n=1$ or $n\geq 2$ is necessarily even. Existence follows from the existence of $\mathrm{BH}(n,4)$ matrices.

\underline{$q=17$}: The Fourier matrix $F_{17}$ was shown to be unique in \cite{cHIR} by computers.

Examples of matrices corresponding to the cases marked by ``E'' in Table~\ref{tableBE} can be obtained from either by viewing a matrix $H\in\mathrm{BH}(n,q)$ as a member of $\mathrm{BH}(n,r)$ with some $r$ which is a multiple of $q$; or by considering the Kronecker product of two smaller matrices \cite[Lemma~4.2]{cHOR}. In particular, if $H\in\mathrm{BH}(n_1,q_1)$ and $K\in\mathrm{BH}(n_2,q_2)$ then $H\otimes K\in\mathrm{BH}(n_1n_2,\mathrm{LCM}(q_1,q_2))$, where $\mathrm{LCM}(a,b)$ is the least commmon multiple of the positive integers $a$ and $b$. This construction shows that Butson matrices of composite orders are abundant. In contrast, very little is known about the prime order case \cite{cPET}.
\begin{rem}\label{HADEQ}
Several authors, see e.g.\ \cite[Definition~4.12]{cHOR}, \cite{cLOS}, consider two $\mathrm{BH}(n,q)$ matrices Hadamard equivalent if either can be obtained from the other by performing a finite sequence of monomial equivalence preserving operations, and by replacing every entry by its image under a fixed automorphism of $\mathbb{Z}_q$. Given the classification of Butson matrices up to monomial equivalence it is a routine task to determine their number up to Hadamard equivalence. Indeed, let $\mathcal{X}$ be a complete set of representatives of $\mathrm{BH}(n,q)$ matrices up to monomial equivalence. Let $\varphi(.)$ denote the Euler's totient function. Then for each $H\in\mathcal{X}$ let us denote by $c(\Psi(H))$ the number of matrices in $\Psi(H):=\{\psi(H)\colon\psi\in\mathrm{Aut}(\mathbb{Z}_q)\}$ up to monomial equivalence. For each $i\in\{1,\dots,\varphi(q)\}$ let us denote by $k_i$ the frequency distribution of the number $i$ occurring as the value of $c(\Psi(H))$ while it runs through $\mathcal{X}$. Then the number of Hadamard equivalence classes is $\sum_{i=1}^{\varphi(q)}k_i/i$, see Table~\ref{tableHEQ}.
\end{rem} 
\subsection{Classification of the BH(16,4) matrices}\label{seccase1}
Classification of the quaternary complex Hadamard matrices is motivated by their intrinsic connection to real Hadamard matrices, which is best illustrated by the following classical result.

\begin{theorem}[\cite{cCHK}, \cite{cTUR}]\label{turync}
Let $n\geq 1$. If $A$ and $B$ are $n\times n$ $\{-1,0,1\}$-matrices such that $A+\mathbf{i}B\in\mathrm{BH}(n,4)$ then $A\otimes\left[\begin{smallmatrix}1 &\hfill 1\\ 1 &\hfill -1\end{smallmatrix}\right]+B\otimes\left[\begin{smallmatrix}\hfill-1 & 1\\\hfill 1 & 1\end{smallmatrix}\right]\in\mathrm{BH}(2n,2)$.
\end{theorem}

It is conjectured \cite[p.~68]{cHOR} that $\mathrm{BH}(n,4)$ matrices exist for all even $n$. The resolution of this ``Complex Hadamard Conjecture'' would imply by Theorem~\ref{turync} the celebrated Hadamard Conjecture.

The classification of $\mathrm{BH}(16,4)$ matrices involved several steps. First we generated the set $\mathcal{O}(16,4)$. We note that $|\mathcal{O}(16,4)|=8$ by Lemma~\ref{minor1}, and these elements can be obtained from Lemma~\ref{l1} by simple hand calculations. Then, we broke up the task of classification into $5$ smaller subproblems of increasing difficulty based on the presence of certain substructures. This allowed us to experiment with the simpler cases and to develop and test algorithms used for the more involved ones. In the following we introduce the type of a $\mathrm{BH}(n,4)$ matrix, a concept which is invariant up to monomial equivalence. A similar idea was used during the classification of $\mathrm{BH}(32,2)$ matrices \cite{cKha}.
\begin{definition}
Let $n,r\geq 2$, let $R$ be an $r\times n$ orthogonal matrix with $4$th root entries, and let $r_1$ and $r_2$ be distinct rows of $L(R)$. Let $m$ denote the number of $0$ entries in the difference vector $r_1-r_2\in\mathbb{Z}_4^n$, and let $k:=\min\{m,n/2-m\}$. Then the subset of rows $\{r_1,r_2\}$ is said to be of type-$k$. The matrix $R$ is said to be of type-$k$, if $L(R)$ has no two rows which are of type-$\ell$ for any $\ell<k$.
\end{definition}
Secondly, we fixed $k\in\{0,\dots,4\}$ and generated the $5\times 16$ canonical (see Section~\ref{subsorder}) type-$k$ matrices surviving the second column pruning strategy. Thirdly, we augmented each of these with three additional rows to obtain all $8\times 16$ matrices, but during this process a depth-first-search approach was employed, and the $r\times 16$ submatrices were not kept for $r\in\{6,7\}$. Finally, we finished the search by using breadth-first-search to generate all $r\times 16$ matrices step-by-step for each $r\in\{9,\dots,16\}$. The reader is invited to compare the size of the search trees involved with the $\mathrm{BH}(16,2)$ case displayed in Table~\ref{table2} and with the $\mathrm{BH}(14,4)$ case displayed in Table~\ref{tabletreecomp}.

\tiny
\begin{table}[htbp]
\begin{tabular}{rrrrrrr}
\hline
 r & \text{Type-0} & \text{Type-1} & \text{Type-2} & \text{Type-3} & \text{Type-4} & $\mathrm{BH}(16,2)$\\
\hline
 2 &            1 &            1 &            1 &            1 &            1      &  1\\
 3 &            9 &            9 &           49 &           26 &           10      &  1\\
 4 &         1397 &         8633 &        56097 &        32893 &         1679      &  3\\
 5 &      1194940 &      7100100 &     45512519 &     14340921 &       193820      &  2\\
 6 &    110431982 &    334154285 &   1739437037 &    250825832 &       784744      &  3\\
 7 &    376589253 &    529596667 &   2085549171 &    126133829 &        95814      &  4\\
 8 &     45784720 &     30437221 &     78690938 &      1960798 &         1088      &  4\\
 9 &     88353309 &     29707820 &     49967830 &       521903 &          260      &  4\\
10 &    123354601 &     24749147 &     28354094 &       132072 &          188      &  7\\
11 &    131598863 &     17398376 &     14649819 &        30142 &           70      &  7\\
12 &    102432783 &     10364363 &      6091931 &         6600 &           21      & 15\\
13 &     56174515 &      4729081 &      1987727 &         1477 &           48      &  8\\
14 &     23306156 &      1981269 &       739324 &          778 &           57      &  8\\
15 &      6999913 &       579250 &       246614 &          327 &           22      &  5\\
16 &      1599355 &       136583 &        50704 &          106 &           15      &  5\\
\hline
\end{tabular}
\caption{Comparison of the search trees of the $\mathrm{BH}(16,4)$ and $\mathrm{BH}(16,2)$ cases.}\label{table2}
\end{table}
\normalsize

The search, which relied on only the standard \CC{} libraries and an army of 896 computing cores took more than $30$ CPU years, and yielded the following classification result.
\begin{theorem}
The number of $\mathrm{BH}(16,4)$ matrices is $1786763$ up to monomial equivalence.
\end{theorem}
In Table~\ref{tableautgrps} we exhibit the automorphism group sizes along with their frequencies.

\tiny
\begin{table}[htbp]
\begin{tabular}{rrrrrrrr}
\hline
 $|\mathrm{Aut}|$ & \text{\#} & $|\mathrm{Aut}|$ & \text{\#} & $|\mathrm{Aut}|$ & \text{\#} & $|\mathrm{Aut}|$ & \text{\#}\\
\hline
 20643840 & 1 & 12288 & 12 & 1024 & 863 & 96 & 594 \\
 589824 & 1 & 8192 & 54 & 768 & 94 & 64 & 67186 \\
 196608 & 1 & 6144 & 16 & 512 & 2410 & 56 & 6 \\
 172032 & 4 & 4096 & 74 & 448 & 2 & 48 & 820 \\
 98304 & 4 & 3840 & 1 & 384 & 212 & 32 & 204627 \\
 65536 & 1 & 3584 & 1 & 336 & 2 & 28 & 6 \\
 49152 & 6 & 3072 & 47 & 320 & 2 & 24 & 706 \\
 36864 & 2 & 2688 & 3 & 256 & 6112 & 16 & 406213 \\
 24576 & 6 & 2048 & 266 & 192 & 260 & 12 & 141 \\
 21504 & 2 & 1536 & 64 & 128 & 18540 & 8 & 554877 \\
 16384 & 10 & 1280 & 2 & 112 & 6 & 4 & 522506 \\
\hline
\end{tabular}
\caption{The automorphism group sizes of $\mathrm{BH}(16,4)$ matrices.}\label{tableautgrps}
\end{table}
\normalsize

\begin{corollary}
The total number of $\mathrm{BH}(16,4)$ matrices $($not considering equivalence$)$ is exactly $1882031756845055238646027031522819126506763059200000$.
\end{corollary}
\begin{proof}
Let $\mathcal{X}$ be a a complete set of representatives of $\mathrm{BH}(16,4)$ matrices up to monomial equivalence. Then the size of the set $\mathrm{BH}(16,4)$ can be inferred from an application of the Orbit-stabilizer theorem \cite[Theorem~3.20]{cKAS}. We have $|\mathrm{BH}(16,4)|=|G|\sum_{X\in\mathcal{X}}1/|\mathrm{Aut}(X)|$. Combining $|G|=(16!)^2\cdot4^{32}$ with the numbers shown in Table~\ref{tableautgrps} yields the result.
\end{proof}
There are two main reasons for the existence of such a huge number of equivalence classes. First, Kronecker-like constructions can lift up the $\mathrm{BH}(8,4)$ matrices resulting in multi-parametric families of complex Hadamard matrices \cite{cDIT}, \cite{cKarol}. The second reason is the presence of type-$0$ (that is: real) pair of rows. It is known that such a substructure can be ``switched'' \cite{cORR} in a continuous way \cite{cSZFPAR} thus escaping the monomial equivalence class of the matrices is possible. In contrast, matrices which cannot lead to continuous parametric families of complex Hadamard matrices are called isolated \cite{cKarol}. A notion to measure the number of free parameters which can be introduced into a given matrix is the defect \cite{cKarol}, which serves as an upper bound. We remark that when $q\in\{2,3,4,6\}$ then computing the defect boils down to a rank computation of integer matrices which can be performed efficiently using exact integer arithmetic.
\begin{corollary}
There are at least $7978$ isolated $\mathrm{BH}(16,4)$ matrices.
\end{corollary}
\begin{proof}
This is established by counting the number of $\mathrm{BH}(16,4)$ matrices with defect $0$. There are no isolated $\mathrm{BH}(16,4)$ matrices of type-$0$, because they contain a real pair of rows as a substructure. It is easy to see that such matrices cannot be isolated once the size of the matrices $n>2$, see \cite{cSZFPAR}. Computation reveals that there are no type-$k$ matrices with vanishing defect for $k\in\{1,3,4\}$, and there are exactly $7978$ type-$2$ matrices with defect $0$. Since the defect is an upper bound on the number of smooth parameters which can be introduced \cite{cKarol}, these matrices are isolated.
\end{proof}
Finally, we note a result connecting $\mathrm{BH}(2n,4)$ matrices with unreal $\mathrm{BH}(n,6)$ matrices.
\begin{theorem}\label{newkron}
If $A$ and $B$ are $n\times n$ $\{-1,0,1\}$-matrices such that $A_{ij}B_{ij}=0$ for $i,j\in\{1,\hdots,n\}$, and $H:=A\omega+B\omega^2\in \mathrm{BH}(n,6)$ with $\omega=\mathrm{exp}(2\pi\mathbf{i}/3)$, then $K:=A\otimes\left[\begin{smallmatrix}1 &\hfill 1\\ 1 &\hfill -1\end{smallmatrix}\right]+B\otimes\left[\begin{smallmatrix}\hfill\mathbf{i}&\hfill-1\\\hfill-1&\hfill\mathbf{i}\end{smallmatrix}\right]\in\mathrm{BH}(2n,4)$.
\end{theorem}
\begin{proof}
Let $X:=\left[\begin{smallmatrix}1 &\hfill 1\\ 1 &\hfill -1\end{smallmatrix}\right]$ and $Y:=\left[\begin{smallmatrix}\hfill\mathbf{i}&\hfill-1\\\hfill-1&\hfill\mathbf{i}\end{smallmatrix}\right]$. We have $XX^\ast=YY^\ast=-(XY^\ast+YX^\ast)=2I_2$. Since $(A\omega+B\omega^2)(A^T\omega^2+B^T\omega)=n I_n$, we have $AB^T=BA^T$. Every entry of $K$ is some $4$th root of unity, and $KK^\ast=(AA^T+BB^T)\otimes (2I_2)+AB^T\otimes(XY^\ast+YX^\ast)=2nI_{2n}$.
\end{proof}
The significance of this observation is that it implies the following recent result.
\begin{corollary}[\cite{cCCdL}]
Let $n\geq 1$ be an integer. If there exists a $\mathrm{BH}(n,6)$ matrix with no $\pm1$ entries, then there exists a $\mathrm{BH}(4n,2)$.
\end{corollary}
\begin{proof}
Combine Theorem~\ref{turync} with Theorem~\ref{newkron}.
\end{proof}

\subsection{Classification of BH(21,3) matrices}\label{bh213x}
In this section we briefly report on our computational results regarding the $\mathrm{BH}(21,3)$ matrices. The classification of $\mathrm{BH}(18,3)$ matrices was reported earlier in \cite{cHAR2} and independently in \cite{cLAM}, while several examples of $\mathrm{BH}(21,3)$ matrices were reported in \cite{cAKI}.

The major difference between this case and the case of $\mathrm{BH}(16,4)$ matrices discussed in Section~\ref{seccase1} is that due to the lack of building blocks (such as a $\mathrm{BH}(7,3)$) for Kronecker-like constructions here one does not expect many solutions to be found, and therefore one may try to approach this problem by employing slightly different techniques.

First, we classified all $r\times 21$ orderly-generated rectangular orthogonal $3$rd root matrices with the second column pruning technique, and found exactly $1$, $1$, $12$, $145$, and $74013$ such matrices up to monomial equivalence for $r\in\{1,2,\dots,5\}$. After this, we considered each of these $5\times 21$ starting-point matrices, say $R$, one-by-one, and generated a set $V$ containing those row vectors which are lexicographically larger than the $5$th row of $R$, and which are orthogonal to each $5$ rows of $R$. Then, following ideas used in \cite{cSPE}, we created the compatibility graph $\Gamma(R)$ on $|V|$ vertices, where two vertices, say $x$ and $y$, indexed by elements of $V$, are adjacent if and only if the rows $x\in V$ and $y\in V$ are pairwise orthogonal. With this terminology the task was then to decide if $\Gamma(R)$ contains a clique of size $16$. It turned out that in most cases it does not, and therefore we could reject the matrix $R$. The Cliquer software \cite{cNIS}, based on \cite{cOS}, was used in the current work to prune inextendible matrices in this way.

It was estimated that around $500$ CPU years is required to solve this case \cite{cHAR2}. However, we have completed this task in just over $18$ CPU days.

\begin{theorem}
The number of $\mathrm{BH}(21,3)$ matrices is $72$ up to monomial equivalence.
\end{theorem}

In Table~\ref{tableautgrps2} we display the automorphism group sizes along with their frequencies.

\tiny
\begin{table}[htbp]
\begin{tabular}{rrrrrrrrrr}
\hline
 $|\mathrm{Aut}|$ & \text{\#} & $|\mathrm{Aut}|$ & \text{\#} & $|\mathrm{Aut}|$ & \text{\#} & $|\mathrm{Aut}|$ & \text{\#} & $|\mathrm{Aut}|$ & \text{\#}\\
\hline
 1008 & 2 & 504 & 4 & 54 & 8 & 36 & 10 & 18 & 12 \\
 720  & 2 &  72 & 8 & 48 & 6 & 24 & 12 & 12 & 8\\
 \hline
\end{tabular}
\caption{The automorphism group sizes of $\mathrm{BH}(21,3)$ matrices.}\label{tableautgrps2}
\end{table}
\normalsize

\subsection{Nonexistence results}
Nonexistence results for Butson matrices were obtained in \cite{cBAN}, \cite{cBRO}, \cite{delaun1}, \cite{cLL}, \cite{cWIN}. To the best of our knowledge the results presented in this section are not covered by any of these previous theoretical considerations.

In this section we briefly report on several exhaustive computational searches which did not yield any Butson matrices. Most of these computations were done in two different ways. First, we established nonexistence by using Cliquer \cite{cNIS}, which heavily pruned the search tree, that is reduced the number of cases to be considered. This was very efficient due to the lack of complete matrices. Once nonexistence was established, we verified it during a second run, but this time without relying on Cliquer. This was done in order to be able to prudently document the search, and to avoid the use of external libraries.
\begin{theorem}\label{nonex8}
There does not exist a $\mathrm{BH}(8,15)$ matrix.
\end{theorem}
\begin{proof}
The proof is computational. We have generated the $r\times 8$ orthogonal matrices with $15$th root of unity entries with the orderly algorithm using the second column pruning strategy, and we found $1$, $1$, $6$, and $0$ such matrices for $r\in\{1,2,3,4\}$, respectively. Therefore there exist no $\mathrm{BH}(8,15)$ matrices.
\end{proof}
\begin{theorem}\label{nonex11}
There does not exist a $\mathrm{BH}(11,q)$ matrix for $q\in\{10,12,14,15\}$.
\end{theorem}
\begin{proof}
The proof is along the lines of the Proof of Theorem~\ref{nonex8}. Refer to Table~\ref{tablenonex11} for the number of orderly-generated, rectangular orthogonal $r\times 11$ matrices with $q$th roots of unity (where $q\in\{10,12,14,15\}$) surviving the second column pruning strategy. In each of the four cases no such matrices were found for some $r\in\{1,\dots,11\}$, hence $\mathrm{BH}(11,q)$ matrices do not exist. For comparison, the case $\mathrm{BH}(11,6)$ is also presented.
\end{proof}

\tiny
\begin{table}[htbp]
\[\begin{array}{r|rrrrrrrrrr}
\hline
r && \mathrm{BH}(11,6) && \mathrm{BH}(11,10) && \mathrm{BH}(11,12) && \mathrm{BH}(11,14) && \mathrm{BH}(11,15)\\
\hline
   1 &&     1 && 1 &&          1 &&    1 &&    1\\
	 2 &&     5 && 5 &&         32 &&    4 &&    3\\
	 3 &&   499 && 0 &&     168564 && 2091 &&  584\\
	 4 && 33655 &&   &&    7950174 && 2572 &&   94\\
   5 && 42851 &&   &&     561071 &&   14 &&   22\\
   6 &&   171 &&   &&        578 &&    0 &&    0\\
   7 &&     0 &&   &&          0 &&      &&     \\
	\hline
\end{array}\]
\caption{The nonexistence of $\mathrm{BH}(11,q)$ matrices for various $q$.}\label{tablenonex11}
\end{table}
\normalsize

\begin{theorem}\label{nonex13}
There does not exist a $\mathrm{BH}(13,10)$ matrix.
\end{theorem}
\begin{proof}
First, we classified the $r\times 13$ orthogonal $10$th root matrices surviving the second column pruning strategy for $r\in\{1,2,3\}$, and found $1$, $10$, and $127556$ such matrices, respectively. As a second step, we used Cliquer \cite{cNIS} to see if any of these $3\times 13$ starting-point matrices can be completed to a $\mathrm{BH}(13,10)$. This task took 250 CPU days, but unfortunately no complete matrices turned up during the search. We note that the number of $4\times 13$ matrices with the relevant properties is exactly $45536950$, and millions of $5\times 13$ and hundreds of $6\times 13$ matrices were found during an incomplete search.
\end{proof}

\section{Open problems}\label{sect99}
We conclude the paper with the following problems.
\begin{problem}
Extend Table~\ref{tableBE} further by classifying some of the remaining cases of $\mathrm{BH}(n,q)$ matrices in the range $n\leq 21$ and $q\leq 17$, and possibly beyond.
\end{problem}
Continue the classification of real Hadamard matrices by extending the work \cite{cORR}.
\begin{problem}
Classify all $\mathrm{BH}(36,2)$ matrices. Is it true that every $H\in\mathrm{BH}(36,2)$ has an equivalent form with constant row sum?
\end{problem}
For context regarding Problem~\ref{sscj} we refer the reader to \cite{cKMat}.
\begin{problem}[Spectral Set Conjecture in $\mathbb{R}^2$]\label{sscj}
Let $n$ and $q$ be positive integers, such that $n\nmid q^2$. Are there rectangular matrices $A$ and $B$ with elements in $\mathbb{Z}_q$ of size $n\times 2$ and $2\times n$, respectively, such that $L(H)=AB$ (modulo $q$) for some $H\in\mathrm{BH}(n,q)$?
\end{problem}
For context regarding Problem~\ref{asym} we refer the reader to \cite{cLL} (see also Remark~\ref{magicsumq}).
\begin{problem}[\mbox{cf.~\cite[Conjecture~7.6]{cBAN}}]\label{asym}
Let $n,q\geq 2$, let $H\in\mathrm{BH}(n,q)$, and let $r_1,r_2\in\mathbb{Z}_q^n$ be distinct rows of $L(H)$. Can $r_1-r_2\in\mathbb{Z}_q^n$ represent an ``asymmetric'' minimal $n$-term vanishing sum of $q$th roots of unity? In other words, is it possible that $\mathrm{Sort}(r_1-r_2)$ is minimal in the sense that it has no constituent of $m$-term vanishing subsums for $m<n$, yet it is not of the form $[0,1,\dots, p-1]\in\mathbb{Z}_q^n$ where $p$ is some prime divisor of $q$?
\end{problem}
Several $\mathrm{BH}(n,q)$ matrices with large $n$ and $q$ were constructed in \cite{cPET}, leading to infinite, parametric families of complex Hadamard matrices of prime orders for $n\equiv 1\ (\mathrm{mod}\ 6)$.
\begin{problem}
Find new examples of $\mathrm{BH}(n,q)$ matrices of prime orders $n\equiv 5\ (\mathrm{mod}\ 6)$.
\end{problem}
Problem~\ref{nextproblem} asks if a non-Desarguesian projective plane of prime order $p$ exists \cite{cHIR}.
\begin{problem}\label{nextproblem}
Let $p$ be a prime number. Decide the uniqueness of $F_p\in\mathrm{BH}(p,p)$.
\end{problem}
The next problem asks for the classification of $q$th root mutually unbiased bases \cite{cKA}.
\begin{problem}
Let $n,q\geq 2$, and let $H,K\in\mathrm{BH}(n,q)$. Classify all pairs $(H,K)$ for which $(HK^{\ast})/\sqrt{n}\in\mathrm{BH}(n,q)$.
\end{problem}

\appendix
\section{Butson matrices up to Hadamard equivalence}
Compare Table~\ref{tableBE} with Table~\ref{tableHEQ} and see Remark~\ref{HADEQ}.
\tiny
\begin{table}[htbp]
\begin{tabular}{r|rrrrrrrrrrrrrrrr}
\hline
${}_n\mkern-6mu\setminus\mkern-6mu{}^q$%
   &         2 &        3   &         4    &         5 &     6       &         7 &         8   &          9 &          10 &        11 &          12 &        13 &           14 &         15 &           16 & 17\\
\hline
 2 & 1 &   &   1 &   &  1 &    &     1 &    &   1 &  &   1 &  &    1 &  & 1 &  \\
 3 &   & 1 &     &   &  1 &    &       &  1 &     &  &   1 &  &      & 1 &  &  \\
 4 & 1 &   &   2 &   &  2 &    &     3 &    &   2 &  &   4 &  &    2 &  & 4 &  \\
 5 &   &   &     & 1 &  0 &    &       &    &   1 &  &   0 &  &      & 1 &  &  \\
 6 & 0 & 1 &   1 &   &  4 &    &     3 &  1 &   0 &  &  10 &  &    0 & 1 & 4 &  \\
 7 &   &   &     &   &  1 &  1 &       &    &   0 &  &   2 &  &    1 &  &  &  \\
 8 & 1 &   &  15 &   & 35 &    &   134 &    & 136 &  & 629 &  &  366 & 0 & 1224 &  \\
 9 &   & 2 &     &   & 10 &    &       & 10 &   1 &  &  33 &  &    0 & 22 &  &  \\
10 & 0 &   &   8 & 1 & 33 &    &    43 &    &  29 &  & 448 &  &    0 & 1 & 124 &  \\
11 &   &   &     &   &  0 &    &       &    &   0 & 1 &  0 &  &    0 & 0 &  &  \\
12 & 1 & 1 & 309 &   & 5758 &  & 28361 &  4 &   76085 &  & \text{E} &  & \text{E} & \text{E} & \text{E} &  \\
13 &   &   &     &   & 218 &   &       &    &   0 &  & \text{E} & 1 & \text{U} & \text{U} &  &  \\
14 & 0 &  & 520 &   & 92325 & 2 & \text{E} &  & \text{E} &  & \text{E} &  & \text{E} & \text{U} & \text{E} &  \\
15 &   & 0 &    & 0 &     0 &   &          & 0 & \text{U} &  & \text{U} &  & \text{U} & \text{E} &  &  \\
16 & 5 &   & 1111624 &  & \text{E} &  & \text{E} &  & \text{E} &  & \text{E} &  & \text{E} & \text{U} & \text{E} &  \\
17 &   &    &  &  & 0 &  &  &  & \text{U} &  & \text{U} &  & \text{U} & \text{U} &  & 1 \\
18 & 0 & 53 & \text{E} &  & \text{E} &  & \text{E} & \text{E} & \text{U} &  & \text{E} &  & \text{U} & \text{U} & \text{E} &  \\
19 &   &    &  &  & \text{E} &  &  &  & \text{U} &  & \text{E} &  & \text{U} & \text{U} &  &  \\
20 & 3 &    & \text{E} & \text{E} & \text{E} &  & \text{E} &  & \text{E} &  & \text{E} &  & \text{E} & \text{E} & \text{E} &  \\
21 &   & 36 &  &  & \text{E} & 0 &  & \text{E} & \text{U} &  & \text{E} &  & 0 & \text{E} &  &  \\
\hline
\end{tabular}
\caption{The number of $\mathrm{BH}(n,q)$ matrices up to Hadamard equivalence.}\label{tableHEQ}
\end{table}
\normalsize

\end{document}